\documentclass[12pt]{amsart}
\usepackage{amsmath,amssymb,latexsym}
\usepackage {amssymb}
\usepackage {amsmath}
\usepackage {graphicx}

\newcommand{\N}{{\Bbb N}}
\newcommand{\D}{{\Bbb D}}

\def\l{\lambda}

\def\d{\mathcal{D}}
\def\z{\zeta}

\newtheorem{Theorem}{Theorem}[section]

\newtheorem{Corollary}{Corollary}[section]
\newtheorem{Remark}{Remark}[section]

\title{Embeddings of vector-valued Bergman spaces}

\begin{document}

\author{Olivia Constantin}

\address{
Faculty of Mathematics,
University of Vienna,
Oskar-Morgenstern-Platz 1, 1090 Vienna,
Austria}
\email{olivia.constantin@univie.ac.at}

\author{Laura G\u avru\c ta}
\address{
Faculty of Mathematics,
University of Vienna,
Oskar-Morgenstern-Platz 1, 1090 Vienna,
Austria}
\email{laura-elena.gavruta@univie.ac.at}

\thanks{The authors were supported by the FWF project
P 24986-N25. }
\date{}



\begin{abstract}
We show that a dyadic version of the Carleson embedding theorem for the Bergman space extends to vector-valued functions and operator-valued
measures. This is in contrast to a result by Nazarov, Treil, Volberg in the context of the Hardy space. We also discuss some embeddings for analytic 
vector-valued functions.
\end{abstract}
\maketitle

\section{Introduction}
Let $\D$, $\mathbb{T}$ denote the unit disc, respectively the unit circle, in the complex plane.
We denote by $\d$  the family of dyadic arcs  $I\subseteq\mathbb{T}$, i.e. sets of the form
$$
I=\{e^{it}\,:\, \frac{2\pi k}{2^n}\le  t\le \frac{2\pi (k+1)}{2^n}\}, \quad k=0,1, ..., 2^n-1,
$$
for $n\in\N$. For each arc $I$, we denote its length by $l(I)$ and we consider the corresponding Carleson "square"
$Q_I=\{z\in \D\,: \, z/|z|\in I,\ 1- l(I)/2\pi\le |z|\}$. Let $T_I$ denote the upper half of $Q_I$,  i.e. 
$T_I=\{z\in Q_I\,: \, \ |z|<1- l(I)/4\pi\}$. 

Given a positive Borel measure $\mu$ on $\D$, a classical theorem by Carleson states that the Hardy space $H^2$
is continuously contained in $L^2(d\mu)$ if and only if there exists a constant $c>0$ such that $\mu(Q_I)\le c \, l(I)$, for 
all $I\in \mathcal{D}$, or, equivalently,
the embedding operator
$$D f(z):= \sum_{I\in\d} (\frac{1}{l(I)} \int_I f dm)\, \chi_{T_I} (z),\quad z\in\D,$$
is bounded from $L^2(dm)$ to $L^2(d\mu)$, where $m$ denotes the Lebesgue measure on $\mathbb{T}$ (see \cite{ntv}). This last equivalent formulation of the 
result is also called the dyadic Carleson embedding theorem. 

A result by Nazarov, Treil, Volberg \cite{ntv} shows that the dyadic Carleson
embedding theorem does not extend to Hilbert space-valued
functions and operator-measures, while sharp dimension-dependent estimates
for matrix-valued measures are provided in \cite{nptv}. 
More precisely, given an $n$-dimensional Hilbert space $H$ and a positive $n\times n$ matrix-valued Borel measure $\mu$ on $\D$,  
the operator $D$ (defined analogously for vector-valued functions) is bounded from $L^2(H, dm)$ to $L^2(H,d\mu)$ if and only if $\mu$ has finite
Carleson intensity $\|\mu\|_C$ (see \cite{nptv} for the definition) and the following {\it sharp} dimensional estimate holds
\begin{equation}\label{dimdepending}
\|D\|\lesssim C \log n \|\mu\|_C^{1/2}.
\end{equation}
Here $L^2(H, dm)$ denotes the space of $H-$valued functions that are square integrable with respect to $m$ on $\mathbb{T}$,
while  $L^2(H,d\mu)$ represents the space of $H-$valued functions that are square integrable with respect to $\mu$ on $\D$.


This result is one in a series of deep investigations (see \cite{tv,nptv,pott}),  which reveal one and the same
phenomenon:  
natural generalizations of some basic results  from the scalar-valued setting to vector-valued Hardy spaces continue to hold as long as
the "target-space" has finite dimension, with the involved estimates depending on
this dimension in such a way that the results fail in the infinite-dimensional case.
 Relevant examples that illustrate this behaviour  appear in the study of the Riesz projection, 
 respectively of Hankel operators with operator-valued symbols on vector-valued Hardy spaces.
Remarkably enough,  when considering the analogous problems for the Bergman projection,
respectively for Hankel operators on vector-valued Bergman
spaces, one encounters a completely different situation, namely, in this context the scalar results have natural generalizations
to the vector-valued setting that are independent of the dimension of the "target space" (see \cite{aaoc,1}).

In this note we show that, when considering Carleson embeddings for vector-valued Bergman 
spaces, one finds a behaviour consistent with the one described above. We start with the observation that
the disc-analogue of the dyadic Carleson embedding theorem continues to hold for operator-valued
measures. 
Subsequently, we use techniques developed in \cite{aaoc} in order to discuss some embeddings for analytic vector-valued functions.
More precisely, we characterize the nonnegative operator-valued functions $G:\D\rightarrow \mathcal{B}(H)$ for which the 
vector-valued  Bergman spaces with operator-valued B\'ekoll\'e-Bonami weights  are continuously contained in $L^2(H, G\,dA)$, where $A$ denotes
the normalized Lebesgue area measure on $\D$.  The obtained characterizations are independent of dimension, which is mainly
due to the fact that the Carleson condition expressed in terms of Carleson "squares" is equivalent to the analogue condition
for top halves of Carleson "squares". We would like to emphasize that our Carleson embedding theorem for analytic 
functions is equivalent to its dyadic version even when the target space $H$ has infinite dimension.
We refer to \cite{blasco} for related results concerning multipliers on vector valued Bergman spaces.
Finally, we 
present a short application of our embedding theorem to integration operators of Volterra type with operator-valued symbols.

\section{A remark on dyadic Carleson embeddings}
Let us first introduce some notation.
For two real-valued (respectively, positive operator-valued) functions $E_1,E_2$ we write $E_1\sim
E_2$, or $E_1\lesssim E_2$, if there exists a positive constant $k$
independent of the argument such that $\frac{1}{k} E_1\leq E_2\leq k
E_1$, respectively $E_1\leq k E_2$.

We denote by $A$ the normalized Lebesgue area measure on $\D$.
Let $H$ be a Hilbert space and consider a positive operator-valued Borel measure $\mu$ on $\D$, i.e.
a countably additive function defined on the Borel measurable subsets of $\D$ with values in the set of
non-negative operators on $H$. 
We define its Carleson intensity as follows
$$\|\mu\|_C:=\sup \frac{\|\mu(Q_I)\|}{A(Q_I)}=\sup_{I\in\d} \sup_{e\in H, \|e\|=1} \frac{\langle\mu(Q_I)e,e\rangle}{A(Q_I)}$$ 
 We consider the following analogue for the disc of the embedding operator $D$ defined above
 $$
 B f(z):= \sum_{I\in\d} \Bigl(\frac{1}{A(T_I)} \int_{T_I} f dA\Bigr)\, \chi_{T_I} (z), \quad z\in\D,\quad f\in L^2(H, dA).
 $$ 
We are concerned with the boundedness of $B$ from $L^2(H, dm)$ to $L^2(H, d\mu)$.
As pointed out in \cite{ntv}, defining the integral $\int_\D \langle d\mu(z) f(z), f(z)\rangle$ for an operator-valued
measure $\mu$ is, in general, a delicate issue.  However, since $B f$ is a vector-valued "step-function", one has
$$\int_\D \langle d\mu(z) (Bf)(z),(Bf)(z)\rangle=
\sum_{I\in\d} \langle \mu(T_I) (\frac{1}{A(T_I)} \int_{T_I} f dA ), (\frac{1}{A(T_I)} \int_{T_I} f dA)\rangle.
$$
The classical Carleson embedding theorem for the scalar valued Bergman space $L_a^2$ states 
that, given a positive scalar Borel measure $\mu$ on $\D$, the space $L_a^2$
is continuously contained in $L^2(d\mu)$ if and only if $\|\mu\|_C<\infty$, or, equivalently the operator
$B$ is bounded from $L^2(dA)$ to $L^2(d\mu)$.

In contrast to the dimension-depending sharp estimate (\ref{dimdepending}) proven  in \cite{nptv}, 
it turns out that, in our setting, we have

\begin{Remark}\label{Dyadic} Assume $dim\, H\le \infty$.  The operator 
$B$ is bounded from $L^2(H, dA)$ to $L^2(H, d\mu)$ if and only if  $\mu$ has finite Carleson intensity.
In addition, there are absolute constants $k_1,k_2>0$ such that the following inequalities hold
$$
k_1\|\mu\|^{1/2}_C\le \|B\|\le k_2 \|\mu\|^{1/2}_C.
$$
\end{Remark}

\begin{proof} The proof is rather straight-forward, but we include it for the sake of completeness.
Assume first that $\|\mu\|_C<\infty$. Notice that $\mu(T_I)\le \mu (Q_I)$ and that $A(T_I)\sim A(Q_I)$, for $I\in\d$. Then, for $f\in L^2(H, dA)$, we have
\begin{eqnarray*}
\int_\D \langle d\mu(z) (Bf)(z),(Bf)(z)\rangle&=&
\sum_{I\in\d} \langle \mu(T_I) (\frac{1}{A(T_I)} \int_{T_I} f dA ), (\frac{1}{A(T_I)} \int_{T_I} f dA)\rangle\\
&\lesssim& \|\mu\|_C \sum_{I\in\d} \frac{1}{A(T_I)}\Bigl\| \int_{T_I} f dA\Bigr\|^2\\
&\le& \|\mu\|_C \sum_{I\in\d} \int_{T_I} \|f\|^2 dA \\
&=& \|\mu\|_C  \int_{\D} \|f\|^2 dA,
\end{eqnarray*}
by the Cauchy-Schwarz inequality.

Conversely, suppose $B$ is bounded. For every $e\in H$ and $I\in\d$ we have $B(\chi_{T_I} e)=\chi_{T_I} e$, and 
hence
$$
\langle \mu(T_I) e,e\rangle=\|B(\chi_{T_I} e)\|^2_{L^2(H,d\mu)}\le \|B\|^2 \|\chi_{T_I} e\|^2_{L^2(H,dA)}= \|B\|^2 A(T_I)\|e\|^2,
$$
which implies 
$$
\alpha:=\sup_{I \in \d} \frac{\|\mu(T_I)\|}{A(T_I)}\le\|B\|^2.
$$
As in the scalar case, it is easily seen that $\alpha$ is 
comparable with $\|\mu\|_C$. Indeed, for any $I\in\d$ we can write the Carleson "square" $Q_I$ as an infinite
union of top halves of Carleson "squares" corresponding to the dyadic subintervals of $I$, 
$$
Q_I=\cup_{n=0}^\infty \cup_{k=1}^{2^n} T^k_{l(I)/2^n},
$$
with $A(T^k_{l(I)/2^n})\sim (\frac{l(I)}{2^n})^2$ for $n \in \N$ and $1 \le k \le 2^n$. This implies
\begin{eqnarray*}
\|\mu(Q_I)\|&\le& \sum_{n=0}^\infty \sum_{k=1}^{2^n} \|\mu(T^k_{l(I)/2^n)}\|\le  \alpha \sum_{n=0}^\infty \sum_{k=1}^{2^n} A(T^k_{l(I)/2^n})
\\
&\lesssim&
 \alpha\sum_{n=0}^\infty \sum_{k=1}^{2^n} (\frac{l(I)}{2^n})^2\lesssim  \alpha l(I)^2\lesssim \alpha A(Q_I).
\end{eqnarray*}
The inequality $\| \mu \|_C \gtrsim \alpha$ is trivial.
\end{proof}

\section{Embeddings of vector-valued Bergman spaces}

We begin by introducing the concepts we are working with.
 
Given a separable Hilbert space $H$, we denote by $\mathcal{B}(H)$ the space of bounded linear operators on $H$. 
We say that a nonnegative operator-valued function $G:\mathbb{D}\rightarrow\mathcal{B}(H)$ (i.e. $G(z)$ is a nonnegative 
operator on $H$, a.e. $z\in\D$)  is {\it integrable} on $\D$ 
if for any $x,y\in\mathcal{H}$ the scalar function $$z\mapsto\langle G(z) x,y\rangle$$ is integrable on $\D$ and we have
$$\bigg|\int_{\mathbb{D}}\langle G(z)x,y\rangle dA(z)\bigg|\leq C \|x\|\|y\|,$$ for some constant $C$ independent of $x$ and $y.$
The bounded linear operator defined this way will be denoted by $\displaystyle \int_{\mathbb{D}}G dA.$

For $\eta>-1$, we denote $dA_\eta(z)=(\eta +1)(1-|z|)^\eta dA(z)$. 
We consider operator-valued weights $W:\mathbb{D}\rightarrow\mathcal{B}(\mathcal{H})$ such that 
\begin{enumerate}
\item $W(z)$ is a nonnegative operator that is invertible a.e. $z\in\D$;
\item $(1-|z|^2)^\eta W$ and $(1-|z|^2)^\eta W^{-1}$ are integrable on $\D$;
\item $\displaystyle\int_{\mathbb{D}}WdA_\eta$ is invertible.
\end{enumerate}
For such a weight $W$ the corresponding $L^2$ space on $\mathbb{D}$ is denoted $L^2(H,WdA_\eta)$ and is endowed with the norm $$\|f\|^2=\int_{\mathbb{D}}\langle W(z)f(z),f(z)\rangle dA_\eta(z).$$

We say that the weight $W$  belongs to the so-called B\'ekoll\'e-Bonami class $B_2(\eta)$ ($\eta>-1$) if
$$\sup_{S}\bigg\|\bigg(\frac{1}{A_{\eta}(S)}\int_{S}WdA_{\eta}\bigg)^{1/2}\bigg
(\frac{1}{A_{\eta}(S)}\int_{S}W^{-1}dA_{\eta}\bigg)^{1/2}\bigg\|<\infty,$$ where the supremum is taken over all Carleson squares:\\ $S=\{z=re^{it}:1-h<r<1, |t-\theta|<\pi h\},$ with $h\in(0,1)$, $\theta\in[0,2\pi).$

\noindent For $W\in B_2(\eta)$ the subspace of $L^2(H,WdA_\eta)$ consisting of $H-$valued analytic functions in $\D$ is closed in $L^2(H, W dA_\eta)$ (see \cite{aaoc}). We denote this subspace 
by $L_a^2(H,WdA_\eta).$\\

\noindent We shall also consider the "average" version of $W$ on the discs
$$D_{z,r}=\{\zeta\in\mathbb{D}: |z-\zeta|<r(1-|z|)\},\quad r\in(0,1),$$ 
given by $$[W]_{z,r}:=\frac{1}{A(D_{z,r})}\int_{D_{z,r}}WdA,\quad z\in\mathbb{D}.$$
If $W$ belongs to $B_2(\eta)$ ($\eta>-1$), it follows that $[W]_{z,r}$ is invertible, for all $z\in\D$ (see Proposition 3.2 in \cite{aaoc}).

We are now ready to characterize the nonnegative 
operator-valued functions $G$ for which $L_a^2(H, WdA_\eta)$ is continuously embedded in $L^2(H, G \,dA)$. Our result is actually 
slightly more general, in the sense that we consider embeddings for derivatives of functions in $L_a^2(H, WdA_\eta)$.

\begin{Theorem}\label{analytic}
Assume $G:\D\rightarrow \mathcal{B}(H)$ is an integrable nonnegative operator-valued function,
let $W$ be an operator-valued weight that belongs to $B_2(\eta)$ for some $\eta>-1$ and let $r\in(0,1)$. Then, for any nonnegative integer $n$, 
\begin{equation}\label{carlesonineq}
\int_\D\langle G(z) f^{(n)}(z),f^{(n)}(z)\rangle \, dA(z)\lesssim \int_\D\langle W(z) f(z),f(z)\rangle \, dA_\eta(z), \quad f\in L_a^2(H, WdA_\eta),
\end{equation}
holds if and only if 
\begin{equation}\label{carlesoncond}
\int_{D_{\l,r}} G\, dA  \lesssim (1-|\l|)^{2n} \int_{D_{\l,r}} W\, dA_\eta,
\end{equation}
where the constant involved in the last inequality above is independent of $\l$.
\end{Theorem}

\begin{proof}
Assume (\ref{carlesoncond}) holds. 
For each $z\in \D$, we apply Cauchy's formula to the analytic function $\zeta\mapsto G^{1/2}(z) f(\zeta)$ to obtain the estimate
\begin{eqnarray*}
\langle G(z) f^{(n)}(z), f^{(n)}(z)\rangle &\lesssim& \frac{1}{ (1-|z|)^{2(n+1)} } \int_{D_{z,\beta}} \langle G(z) f(\z), f(\z)\rangle dA(\z),
\end{eqnarray*}
where $\beta\in (0,1)$.
Now integrate the above inequality on $\D$ to deduce
\begin{eqnarray*}
\int_\D\langle G(z) f^{(n)}(z), f^{(n)}(z) \rangle dA(z) &\lesssim& \int_\D 
 (1-|z|)^{-2(n+1)}\int_{D_{z,\beta}} \langle G(z) f(\z), f(\z)\rangle dA(\z) dA(z).
 \end{eqnarray*}
 Notice that, for $0<\beta<\frac{r}{r+1}$, the inclusion $D_{z,\beta}\subseteq\widetilde{D}_{z,r}:=\{\zeta\in\mathbb{D}:|\zeta-z|<r(1-|\zeta|)\}$ holds.
 Using this together with Fubini's theorem in the last relation above we obtain
 \begin{eqnarray*}
\int_\D\langle G(z) f^{(n)}(z), f^{(n)}(z) \rangle dA(z)
&\lesssim& \int_\D 
 (1-|z|)^{-2(n+1)}\int_{\widetilde{D}_{z,r}} \langle G(z) f(\z), f(\z)\rangle dA(\z) dA(z)\\
 &\lesssim& \int_\D \langle \int_{{D}_{\z,r}} G(z) dA(z)\  f(\z), f(\z)\rangle (1-|\z|)^{-2(n+1)}dA(\z) \\
 &\lesssim& \int_\D \langle \int_{{D}_{\z,r}} W dA_\eta\  f(\z), f(\z)\rangle (1-|\z|)^{-2}dA(\z) \\
 &\lesssim& \int_\D \langle [W]_{\z,r}  f(\z), f(\z)\rangle dA_\eta(\z), 
 \end{eqnarray*}
by (\ref{carlesoncond}) and since $(1-|z|)\sim(1-|\z|)$ for $z\in D_{\z,r}$.
Finally, an application of the second part of Theorem 3.1 in \cite{aaoc} now yields
\begin{eqnarray*}
\int_\D\langle G(z) f^{(n)}(z), f^{(n)}(z) \rangle dA(z)\lesssim \int_\D \langle W(\z) f(\z), f(\z)\rangle dA_\eta(\z).
\end{eqnarray*}

Conversely, assume that (\ref{carlesonineq}) holds. For $e\in H$ and $\gamma>\eta$, we put $f(z)=K^\gamma_\l(z) e=\frac{1}{(1-\bar\l z)^{\gamma+2}}e$ in (\ref{carlesonineq}) and use Proposition 3.1 in \cite{aaoc} to get
\begin{equation}\label{ineq1}
\int_\D\langle G(z) e,e\rangle |\partial_z^n K^\gamma_\l(z)|^2\, dA(z)\lesssim  
\int_\D\langle W(z) e, e\rangle |K^\gamma_\l(z)|^2 \, dA_\eta(z)\lesssim (1-|\l|)^{-2\gamma+\eta-2}\langle[W]_{\l,r} e,e\rangle.
\end{equation}
Let $\delta:=\frac{r}{r+2}$. For $|\l|\ge\delta>0$ we have
\begin{eqnarray*}
\int_\D\langle G(z) e,e\rangle |\partial_z^n K^\gamma_\l(z)|^2\, dA(z)&\gtrsim&  \int_{D_{\l,r}} \langle G(z) e,e\rangle \frac{1}{|1-\bar \l z|^{2(\gamma+n+2)}}\, dA(z)\\
&\gtrsim& (1-|\l|)^{-2\gamma-2n-4} \int_{D_{\l,r}} \langle G(z) e,e\rangle \, dA(z).
\end{eqnarray*}
Combining this with (\ref{ineq1}), we obtain (\ref{carlesoncond}) for $|\l|\ge\delta$. It is easy to see that (\ref{carlesoncond}) 
holds for $| \l |<\delta$.
Indeed, 
in order
to prove (\ref{carlesoncond}) for $|z|<\delta$, it is enough to show that
 $$
 \int_{D_{z,r}} G \, dA \lesssim \int_{D_{z,r}} W\, dA, \quad |z|<\delta,
 $$
 or, equivalently,
$$
\| (\int_{D_{z,r}} G \, dA)^{1/2}  (\int_{D_{z,r}} W \, dA)^{-1/2} \|\lesssim 1, \quad |z|<\delta,
$$
where the invertibility of $\int_{D_{z,r}} W \, dA$ follows by Proposition 3.1 in \cite{aaoc}.
Now 
$$
\| (\int_{D_{z,r}} G \, dA)^{1/2} \|\le \| (\int_{\D} G \, dA)^{1/2}  \|.
$$
Notice that, since $ r(1-|z|)>r(1-\delta)=2\delta$, we have $D(0,\delta)\subset D_{z,r}$. Hence
$$
\| (\int_{D_{z,r}} W \, dA)^{-1/2}\|\le \| (\int_{D(0,\delta)} W \, dA)^{-1/2}\|
$$
and the desired inequality is proven.
\end{proof}


\begin{Remark}
Take $W=I$, $n=0$ in Theorem \ref{analytic} and let $\mu=G\,dA$ in Remark \ref{Dyadic}. Notice that, for these choices, condition (\ref{carlesoncond}) is equivalent to 
$\| \mu \|_C <\infty$, which shows that in this case the Carleson embedding from Theorem \ref{analytic} is 
equivalent to its dyadic version provided in Remark \ref{Dyadic}. Indeed, the equivalence of the two conditions 
follows, as in the scalar case, by a purely geometric argument: for each fixed $r \in (0,1)$, we can cover any disc $D_{\l, r}$ 
by a finite number (depending only on $r$) of top halves $T_I$, with $I \in {\mathcal D}$, of comparable diameter, and 
vice-versa.
\end{Remark}



We can now apply Theorem \ref{analytic} to characterize the boundedness of integration operators of Volterra type
acting on vector-valued Bergman spaces with B\'ekoll\'e-Bonami weights.
For an analytic operator-valued function $G:\D\rightarrow \mathcal{B}(H)$, and for analytic functions  $f:\D\rightarrow H$, we define the operator
$$
T_G f(z)=\int_0^z G'(\zeta) f(\zeta) d\z,\quad z\in\D.
$$
\begin{Corollary} Assume $W:\D\rightarrow  \mathcal{B}(H)$ 
is an operator-valued weight that belongs to $B_2(\eta)$ for some $\eta>-1$ and let $r\in(0,1)$.
Then $T_G$ is bounded on $L_a^2(H, WdA_\eta)$ if and only if
\begin{equation}\label{bounded}
\sup_{\l\in\D} (1-|\l|)\|[W]_{\l,r}^{1/2} G'(\l) [W]_{\l,r}^{-1/2}\|<\infty.
\end{equation}
\end{Corollary}

\begin{proof}
For $f\in L_a^2(H, WdA_\eta)$ we apply Theorem 3.2 in \cite{aaoc} to deduce
\begin{eqnarray*}
\|T_G f\|^2&\sim& \int_\D \langle W(z) G'(z)f(z), G'(z)f(z)\rangle (1-|z|^2)^2 \, dA_\eta(z)\\
&=& \int_\D \langle G'^*(z) W(z)G'(z)f(z), f(z)\rangle (1-|z|^2)^2 \, dA_\eta(z).
\end{eqnarray*}
By Theorem \ref{analytic} it follows that $T_G$ is bounded on $L_a^2(H, WdA_\eta)$ if and only if
$$
\int_{D_{\l,r}} G'^*(z) W(z)G'(z) (1-|z|^2)^2\, dA_\eta(z)\lesssim \int_{D_{\l,r}} W(z)\, dA_\eta(z).
$$
As shown in Theorem 3.1 in \cite{aaoc}, the weight $z\mapsto [W]_{z,r}$ provides an equivalent norm for $L_a^2(H, WdA_\eta)$,
and hence we can replace $W$ by $[W]_{z,r}$ above to obtain that $T_G$ is bounded on $L_a^2(H, WdA_\eta)$ if and only if
\begin{equation}\label{bdd}
\int_{D_{\l,r}} G'^*(z) [W]_{z,r} G'(z) (1-|z|^2)^2\, dA_\eta(z)\lesssim \int_{D_{\l,r}} [W]_{z,r}\, dA_\eta(z),\quad \l\in\D.
\end{equation}
According to Remark 3.2 in \cite{aaoc}, we have 
\begin{equation}\label{comp}
k_1[W]_{\l,r}\le [W]_{z,r}\le k_2 [W]_{\l,r},\quad z\in D_{\l,r},
\end{equation}
where the constants $k_1,k_2>0$ are independent of $z,\l\in\D$. From this we deduce that (\ref{bdd}) is in its turn equivalent to
\begin{equation*}
\int_{D_{\l,r}} G'^*(z) [W]_{\l,r} \,G'(z)\, dA(z)\lesssim  [W]_{\l,r},\quad \l\in\D.
\end{equation*}
Notice that the last relation above can be written as
\begin{equation}\label{bdd1}
\int_{D_{\l,r}} \| [W]^{1/2}_{\l,r} \,G'(z)\, [W]^{-1/2}_{\l,r} e\|^2\, dA(z)\lesssim  \|e\|^2,\quad e\in H,\  \l\in\D.
\end{equation}
We now claim that (\ref{bdd1}) is equivalent to (\ref{bounded}).
Indeed, using the subharmonicity of $z\mapsto \| [W]^{1/2}_{\l,r} \,G'(z)\, [W]^{-1/2}_{\l,r} e\| $ in (\ref{bdd1}) we get
\begin{equation}\label{bdd2}
(1-|\l|)^2\| [W]^{1/2}_{\l,r} \,G'(\l)\, [W]^{-1/2}_{\l,r} e\|^2\lesssim  \|e\|^2,\quad e\in H,\ \l\in\D,
\end{equation}
and hence one implication is proven. The converse follows by integrating (\ref{bdd2}) with respect to $\l$ on $D_{\z,r}$ and using relation (\ref{comp}).
\end{proof}

\end{document}